\pdfoutput=1
\RequirePackage{ifpdf}
\ifpdf % We~are running pdfTeX in pdf mode
\documentclass[pdftex]{sigma}
\else
\documentclass{sigma}
\fi

\numberwithin{equation}{section}

\newtheorem{Theorem}{Theorem}[section]
\newtheorem*{Theorem*}{Theorem}

 { \theoremstyle{definition}

\newtheorem{Remark}[Theorem]{Remark} }

\begin{document}
\allowdisplaybreaks

\newcommand{\arXivNumber}{2202.01278}

\renewcommand{\PaperNumber}{047}

\FirstPageHeading

\ShortArticleName{Determinantal Formulas for Exceptional Orthogonal Polynomials}

\ArticleName{Determinantal Formulas\\ for Exceptional Orthogonal Polynomials}

\Author{Brian SIMANEK}

\AuthorNameForHeading{B.~Simanek}

\Address{Baylor University Math Department, Waco, Texas, 76706, USA}
\Email{\href{mailto:Brian_Simanek@baylor.edu}{Brian\_Simanek@baylor.edu}}
\URLaddress{\url{https://www.baylor.edu/math/index.php?id=925780}}

\ArticleDates{Received February 11, 2022, in final form June 17, 2022; Published online June 25, 2022}

\Abstract{We present determinantal formulas for families of exceptional $X_m$-Laguerre and exceptional $X_m$-Jacobi polynomials and also for exceptional $X_2$-Hermite polynomials. The formulas resemble Vandermonde determinants and use the zeros of the classical orthogonal polynomials.}


\Classification{42C05; 33C47}

\section{Introduction}

The classical families of orthogonal polynomials known as Laguerre, Jacobi, and Hermite polynomials have been studied extensively over the past 200 years because of their applicability and convenient formulas. Many of their applications are derived from the fact that they are polynomial solutions of explicit second order linear differential equations with polynomial coefficients. Exceptional orthogonal polynomials arise in a similar context: as polynomial solutions of explicit second order linear differential equations with rational coefficients. However, for the exceptional orthogonal polynomial sequences, there are a finite number of natural numbers for which there is no polynomial of that degree.

Exceptional orthogonal polynomials (XOPs) were first introduced in \cite{GKM09}. In \cite{GGM}, it was shown that every sequence of XOPs is related to a classical orthogonal polynomial sequence by a finite number of Darboux transformations. This allows us to label every sequence of XOPs as either Laguerre, Jacobi, or Hermite. It also naturally leads us to wonder what properties of the classical polynomials remain valid for the XOPs. For example, the classical orthogonal polynomials are known to satisfy three-term recurrence relations and one could ask if there is an analogous result for XOPs. Indeed, Dur\'an showed in \cite{D15,D21} that exceptional orthogonal polynomial sequences also satisfy recurrence relations, though of a different form (see also \cite{GKKM16}).

The zeros of exceptional orthogonal polynomials present especially interesting phenomena that have inspired considerable study. These zeros divide into two collections: the regular zeros and the exceptional zeros. The regular zeros exhibit behavior that mimics the behavior of the zeros of the classical orthogonal polynomials, while the exceptional zeros do not. Due to the presence of these exceptional zeros, the zeros of the classical orthogonal polynomials are easier to understand and study than the zeros of the XOPs. Detailed information about the zero behavior of exceptional orthogonal polynomial sequences can be found in \cite{Bonn,BK18,GMM,KM,SimExcep}.

Our analysis will focus on determinantal formulas for certain families of XOPs. There are ways to express the classical orthogonal polynomials using determinants of matrices whose entries are moments of the measure of orthogonality. A comparable result for the most elementary families of XOPs follows from \cite{KLOa,KLO,LO}. Dur\'an found alternative determinantal formulas for XOPs in \cite{D14,D14a,D17} by considering limits of discrete XOPs (see also \cite{D21a}). Our main results will be formulas for exceptional Laguerre, exceptional Jacobi, and exceptional Hermite polynomials as determinants of matrices whose entries are powers of zeros of the corresponding classical orthogonal polynomials (except in the last row, where the entries are explicit polynomials). Consequently, the Vandermonde determinant of the zeros of the classical orthogonal polynomials will appear as a factor in the resulting formulas.

In Section \ref{formulas}, we will review the basic facts and formulas about classical and exceptional orthogonal polynomials that will be relevant for our later calculations. In Section~\ref{newform}, we will derive our new determinantal formulas. In every case, we will first need to prove a first order differential relation for the polynomials (some of which were already known). In the Hermite case, our determinantal formula will only apply to one specific exceptional orthogonal polynomial sequence. For the Laguerre and Jacobi cases, our results will apply to XOP sequences that are obtained from the classical sequences by a single Darboux transformation.

\section{Background and formulas}\label{formulas}

\subsection{Classical orthogonal polynomials}

The starting place for our calculations will be a closer examination of the relationship between the exceptional orthogonal polynomials and the classical orthogonal polynomials. The next several sections will briefly catalog some of the relevant formulas that will be needed in our calculations. All of the formulas below involving the classical families of orthogonal polynomials can be found in the Digital Library of Mathematical Functions \cite{dlmf}.

\subsubsection{Classical Laguerre polynomials}%\label{clag}

For any $\alpha>-1$, the degree $n$ Laguerre polynomial $L^{(\alpha)}_n(x)$ is given by
\begin{equation}\label{lagform}
L_n^{(\alpha)}(x)=\sum_{m=0}^n\frac{(-1)^m}{m!}\binom{n+\alpha}{n-m}x^m.
\end{equation}
These polynomials are orthogonal in the space $L^2\big([0,\infty),x^{\alpha}{\rm e}^{-x}{\rm d}x\big)$. Notice that the leading coefficient of $L_n^{(\alpha)}(x)$ is $(-1)^n/n!\neq0$, so even if $\alpha\leq -1$, the formula~\eqref{lagform} still defines a~poly\-nomial of degree $n$.

There are many interesting and applicable identities involving Laguerre polynomials. We will only need a small collection of these identities in our calculations. The formulas that will be most relevant for us are the following:
\begin{gather}
\label{lagder}
\frac{\rm d}{{\rm d}x}L_n^{(\alpha)}(x)=-L_{n-1}^{(\alpha+1)}(x),
\\[.5ex]
\label{lagform1}
L_n^{(\alpha)}(x)=L_{n}^{(\alpha+1)}(x)-L_{n-1}^{(\alpha+1)}(x),
\\
\label{lagform2}
xL_{n}^{(\alpha+1)}(x)=(n+\alpha+1)L_n^{(\alpha)}(x)-(n+1)L_{n+1}^{(\alpha)}(x).
\end{gather}
Another relevant fact is that the polynomial $L^{(\alpha)}_n$ is the unique polynomial of degree $n$ that solves the ODE
\begin{equation*}%\label{lagode}
xy''+(\alpha+1-x)y'+ny=0
\end{equation*}
with normalization
\begin{equation*}%\label{lagzero}
L_n^{(\alpha)}(0)=\binom{\alpha+n}{n}.
\end{equation*}

We will also need the following result about the zeros of classical Laguerre polynomials.

\begin{Theorem}\label{clagzero}
Suppose $L_n^{(\alpha)}$ is defined by \eqref{lagform}.
\begin{itemize}\itemsep=0pt
\item[$(i)$] If $\alpha>-1$, then the zeros of $L_n^{(\alpha)}$ are simple and lie in the interval $(0,\infty)$.
\item[$(ii)$] If $m\in\mathbb{N}$ and $\alpha>m-1$, then the zeros of $L_m^{(-\alpha-1)}$ are simple.
\item[$(iii)$] If $m\in\mathbb{N}$ and $\alpha>m-1$, then $L_m^{(-\alpha-1)}$ does not have any zeros in $(0,\infty)$.
\end{itemize}
\end{Theorem}

Part $(i)$ of Theorem~\ref{clagzero} is part of a well-known classical result. Part $(ii)$ follows from \cite[equation~(6.71.6)]{SzegoBook} and part $(iii)$ follows from \cite[Theorem~6.73]{SzegoBook}. Note that \cite[Section~6.73]{SzegoBook} states that \cite[equation~(6.71.6)]{SzegoBook} is valid for all $\alpha\in\mathbb{R}$.

\subsubsection{Classical Jacobi polynomials}%\label{cjac}

If $\alpha,\beta>-1$, the degree $n$ Jacobi polynomial $P_n^{(\alpha,\beta)}(x)$ is given by
\begin{equation}\label{jacform}
P_{n}^{(\alpha,\beta)}(x)=\frac{\Gamma(\alpha+n+1)}{n!\Gamma(\alpha+\beta+n+1)} \sum_{m=0}^n\binom{n}{m}\frac{\Gamma(\alpha+\beta+n+m+1)}{\Gamma(\alpha+m+1)}\bigg(\frac{x-1}{2}\bigg)^m.
\end{equation}
These polynomials are orthogonal in the space $L^2\big([-1,1],(1-x)^{\alpha}(1+x)^{\beta}{\rm d}x\big)$. Even if $\alpha,\beta\leq -1$, the formula \eqref{jacform} still defines a polynomial, though it may have degree smaller than $n$ (see \cite[equation~(89)]{GMM}). It is stated in \cite[equation~(6.72.3)]{SzegoBook} that as long as $\alpha+\beta+n\not\in\{-1,-2,\dots,-n\}$, then $P_n^{(\alpha,\beta)}$ has degree $n$.

We will not need much in the way of additional formulas for classical Jacobi polynomials, but we will need the following result about the zeros of classical Jacobi polynomials.

\begin{Theorem}\label{cjaczero}
Suppose $P_n^{(\alpha,\beta)}$ is defined by \eqref{jacform}.
\begin{itemize}\itemsep=0pt
\item[$(i)$] If $\alpha,\beta>-1$, then the zeros of $P_n^{(\alpha,\beta)}$ are simple and lie in the interval $(-1,1)$.
\item[$(ii)$] If $\alpha,\beta\not\in\{-1,-2,\dots,-n\}$ and $\alpha+\beta\not\in\{-1-n,-2-n,\dots,-2n\}$, then the zeros of $P_n^{(\alpha,\beta)}$ are simple.
\item[$(iii)$] If $\beta,\alpha+1-m\in(-1,0)$ or $\beta,\alpha+1-m\in(0,\infty)$, then $P_m^{(-\alpha-1,\beta-1)}$ has no zeros in $(-1,1)$.
\end{itemize}
\end{Theorem}

Part $(i)$ of Theorem \ref{cjaczero} is part of a well-known classical result. Part $(ii)$ follows from \cite[equation~(6.71.5)]{SzegoBook} and part $(iii)$ is stated in \cite[Section 5.2]{GMM}. Note that \cite[Section 6.72]{SzegoBook} states that \cite[equation~(6.71.5)]{SzegoBook} is valid for all $\alpha,\beta\in\mathbb{R}$.

\subsubsection{Classical and generalized Hermite polynomials}\label{cherm}

The classical Hermite polynomials can be defined from the classical Laguerre polynomials by the relations
\begin{equation}\label{hermlag}
H_{2n}(x)=(-4)^nn!L_n^{-1/2}\big(x^2\big),\qquad
H_{2n+1}(x)=2(-4)^nn!xL_n^{1/2}\big(x^2\big),
\end{equation}
from which it immediately follows that
\begin{equation}\label{hermder}
\frac{\rm d}{{\rm d}x}H_n(x)=2nH_{n-1}(x),\qquad n\geq1.
\end{equation}
We will also use the fact that
\begin{equation}\label{hermrecur}
H_n(x)=2xH_{n-1}(x)-H_{n-1}'(x).
\end{equation}
The relevant theorem about the zeros of Hermite polynomials is the fact that they are simple and lie on the real line.

To define the generalized Hermite polynomials, one first defines a partition $\lambda=(\lambda_1,\dots,\lambda_m)\allowbreak\in\mathbb{N}^m$, where
\begin{gather*}
\lambda_1\geq\lambda_2\geq\cdots\geq\lambda_m\geq1.
\end{gather*}
The generalized Hermite polynomial $H_{\lambda}$ is then defined by
\begin{gather*}
H_{\lambda}=\operatorname{Wr}[H_{\lambda_m},H_{\lambda_{m-1}+1},\dots,H_{\lambda_2+m-2},H_{\lambda_1+m-1}],
\end{gather*}
where $\operatorname{Wr}$ denotes the Wronskian determinant. It is known that the degree of $H_{\lambda}$ is $|\lambda|=\lambda_1+\cdots+\lambda_m$ (see \cite[Section~2]{KM}). We call the partition $\lambda$ an \textit{even} partition if $m$ is even and $\lambda_{2j}=\lambda_{2j-1}$ for $j=1,\dots,m/2$.
The relevant fact about the zeros of $H_{\lambda}$ is the following result from \cite[Section~2]{KM}.
\begin{Theorem}%\label{chermzero}
If $\lambda$ is an even partition, then $H_{\lambda}$ has no real zeros.
\end{Theorem}
Simplicity of the zeros of $H_{\lambda}$ is not known, but it was conjectured in \cite{FHV} that the non-zero zeros of $H_{\lambda}$ are always simple (even if $\lambda$ is not even).

\subsection{Exceptional orthogonal polynomials}%\label{xop}

Now we are ready to shift our focus to the exceptional orthogonal polynomials. We will refer to an XOP sequence with $m$ missing degrees as an $X_m$ polynomial sequence. These polynomials have been heavily studied and many formulas are known that will be relevant to our calculations. We will review these facts as well as some facts about the zeros of these polynomials.

\subsubsection[Exceptional X\_m-Laguerre polynomials]{Exceptional $\boldsymbol{X_m}$-Laguerre polynomials}%\label{xlag}

Exceptional Laguerre polynomials were originally introduced in \cite{GKM09} and have been studied in papers such as \cite{BK18,D14a,D21,DP,GMM,HOS,KLO,LLMS,LO,Q09,SimExcep}. We will focus on the $X_m$ polynomials that come from the application of a single Darboux transformation to the classical Laguerre differential operator. These polynomials can be of type-I, type-II, or type-III (the more general setting was discussed in \cite{BK18}). While there is an extensive theory associated to these polynomials, we will only discuss the portion that is relevant to our new results and refer the reader to the aforementioned references for additional information.%and we will denote these polynomials by $L_{m,n}^{I,(\alpha)}(x)$,

In the type-I case, we require $\alpha>0$ and $m\in\mathbb{N}$. The degree $n$ polynomial in this sequence is denoted $L_{m,n}^{I,(\alpha)}(x)$ and $n$ can be any natural number strictly larger than $m-1$ (that is, polynomials of degree $0,1,\dots,m-1$ are omitted from this sequence). For such a choice of parameters, it holds that
\begin{equation}\label{xlag1form}
L_{m,n}^{I,(\alpha)}(x)=L_m^{(\alpha)}(-x)L_{n-m}^{(\alpha-1)}(x)+L_{m}^{(\alpha-1)}(-x)L_{n-m-1}^{(\alpha)}(x)
\end{equation}
(see \cite[Proposition~3.1]{GMM} or \cite[equation~(3.2)]{LLMS}). One can construct these polynomials using partitions as in \cite{BK18} by considering the empty partition and the singleton $\{m\}$ (see \cite[Proposition~4]{BK18}).

In the type-II case, we require $m\in\mathbb{N}$ and $\alpha>m-1$. The degree $n$ polynomial in this sequence is denoted $L_{m,n}^{II,(\alpha)}(x)$ and $n$ can be any natural number strictly larger than $m-1$. For such a choice of parameters, it holds that
\begin{equation}\label{xlag2form}
L_{m,n}^{II,(\alpha)}(x)=xL_m^{(-\alpha-1)}(x)L_{n-m-1}^{(\alpha+2)}(x)+(m-\alpha-1)L_{m}^{(-\alpha-2)}(x)L_{n-m}^{(\alpha+1)}(x)
\end{equation}
(see \cite[Proposition~4.1]{GMM}). One can construct these polynomials using partitions as in \cite{BK18} by considering the empty partition and the $m$-tuple $\{1,1,\dots,1\}$ (see \cite[Proposition~4]{BK18}).

In the type-III case, we require $\alpha\in(-1,0)$ and $m\in\mathbb{N}$. The degree $n$ polynomial in this sequence is denoted $L_{m,n}^{III,(\alpha)}(x)$ and $n$ can be $0$ or any natural number strictly larger than $m$. For such a choice of parameters and with $n\geq m+1$, it holds that
\begin{equation}\label{xlag3form}
L_{m,n}^{III,(\alpha)}(x)=xL_{n-m-2}^{(\alpha+2)}(x)L_{m}^{(-\alpha-1)}(-x)+(m+1)L_{m+1}^{(-\alpha-2)}(-x)L_{n-m-1}^{(\alpha+1)}(x)
\end{equation}
(see \cite[equation~(5.12)]{LLMS}). If $n=0$, then $L_{m,0}^{III,(\alpha)}(x)=1$. One can construct these polynomials using partitions as in \cite{BK18} by considering the $m$-tuple $\{1,1,\dots,1\}$ and the empty partition (see \cite[Proposition 4]{BK18}).
The starting point for our investigation comes from the following formula, which is \cite[Theorem 5.2]{LLMS}:
\begin{equation}\label{lag3llms}
L_{m,n}^{III,(\alpha)}(x)=n\int_0^xL_{n-m-1}^{(\alpha+1)}(t)L_m^{(-\alpha-1)}(-t){\rm d}t+(m+1)\binom{n-m+\alpha}{n-m-1}\binom{m-\alpha-1}{m+1},
\end{equation}
when $n>m$.

 In the case $m=1$, Kelly, Liaw and Osborn derived a determinantal formula for exceptional Laguerre polynomials using what they call ``adjusted moments" (see \cite[Theorem 5.1]{KLOa}; see also~\cite{LO}). This was subsequently generalized to $L_{2,n}^{I,(\alpha)}$ in \cite{KLO}. Dur\'an \cite{D14a} also discovered a method for calculating exceptional Laguerre polynomials using determinantal formulas. Our new results for exceptional Laguerre polynomials appear as Theorems~\ref{lag1detform}, \ref{lag2detform} and~\ref{lag3detform}, all of which are determinantal formulas involving the zeros of classical Laguerre polynomials. In the type-I and type-II cases, we will need a result analogous to \eqref{lag3llms} (see Theorem \ref{lagints}), while in the type-III case, we can appeal directly to \eqref{lag3llms}.

\subsubsection[Exceptional X\_m-Jacobi polynomials]{Exceptional $\boldsymbol{X_m}$-Jacobi polynomials}%\label{xjac}

Exceptional Jacobi polynomials were originally introduced in \cite{GKM09} and have been studied in papers such as \cite{Bonn,D17,D21,GGM21,GMM,HOS,KLOa,LLS,Q09,SimExcep}. The polynomials we will consider come from the application of a single Darboux transformation to the classical Jacobi differential operator and are thus eigenfunctions of a second order differential operator (a more general setting was discussed in \cite{Bonn}). They depend on a parameter $m\in\mathbb{N}$ and two real parameters $\alpha$ and $\beta$, so we will denote the degree $n$ polynomial by $P_{m,n}^{(\alpha,\beta)}$ and require $n\geq m$ (that is, polynomials of degree $0,1,\dots,m-1$ are omitted from this sequence). To ensure certain properties of the polynomial degree and the zeros of these polynomials, we will insist that $\beta>0$, that $\alpha+1-m>0$, and that $\alpha+1-m-\beta\not\in\{0,1,\dots,m-1\}$. In particular, the assumption $\alpha+1-m-\beta\not\in\{0,1,\dots,m-1\}$ guarantees that the polynomial $P_m^{(-\alpha-1,\beta-1)}$ has degree $m$ (see \cite[equation~(6.72.3)]{SzegoBook}). Further justification for these assumptions is discussed in \cite{GMM}, but one can already see their importance from Theorem \ref{cjaczero}. Under these assumptions, the following formula is valid:
\begin{gather}
P_{m,n}^{(\alpha,\beta)}(x)=\frac{(-1)^m}{\alpha+1+n-m} \bigg(\frac{(x-1)(1+\alpha+\beta+n-m)}{2}P_m^{(-\alpha-1,\beta-1)}(x)P_{n-m-1}^{(\alpha+2,\beta)}(x)\nonumber
\\ \hphantom{P_{m,n}^{(\alpha,\beta)}(x)=}
{}+(\alpha+1-m)P_m^{(-\alpha-2,\beta)}(x)P_{n-m}^{(\alpha+1,\beta-1)}(x)\bigg),
\label{xjacform}
\end{gather}
when $n\geq m$ (see \cite[Section 5]{GMM}). One can construct these polynomials using partitions as in~\cite{Bonn} by considering the empty partition and the $m$-tuple $\{1,1,\dots,1\}$ (see \cite[equation~(5.4)]{Bonn}).

Dur\'an \cite{D17} has found determinantal formulas for exceptional Jacobi polynomials and the aforementioned results from~\cite{KLOa} also apply to $P_{1,n}^{(\alpha,\beta)}$. Our new results for these polynomials will include an integral formula comparable to~\eqref{lag3llms} (see Theorem \ref{jacint}) and a determinantal formula for the exceptional Jacobi polynomials that involves the zeros of the classical Jacobi polynomials (see Theorem~\ref{jacdetform}).

\subsubsection[Exceptional X\_m-Hermite polynomials]{Exceptional $\boldsymbol{X_m}$-Hermite polynomials}%\label{xherm}

Exceptional Hermite polynomials were originally introduced in \cite{GGM14} and have been studied in papers such as \cite{D14,D15,GKKM16,KM,SimExcep}. The polynomials we will consider depend on a choice of partition~$\lambda$ as in Section~\ref{cherm}.
Given such a partition $\lambda=(\lambda_1,\dots,\lambda_m)$, one defines the exceptional Hermite polynomial $H_{|\lambda|,n}^{(\lambda)}$ by
\begin{gather*}
H_{|\lambda|,n}^{(\lambda)}=\operatorname{Wr}[H_{\lambda_m},H_{\lambda_{m-1}+1},\dots, H_{\lambda_2+m-2},H_{\lambda_1+m-1},H_{n-|\lambda|+m}]
\end{gather*}
for $n\geq|\lambda|-m$. Even for such $n$, $H_{|\lambda|,n}^{(\lambda)}$ is identically zero unless $n$ is in the set $\mathbb{N}_{\lambda}$ defined by
\begin{gather*}
\mathbb{N}_{\lambda}=\{n\geq|\lambda|-m\colon n\neq|\lambda|+\lambda_j-j,\ \text{for}\ j=1,2,\dots,m\}.
\end{gather*}
If $n\in\mathbb{N}_{\lambda}$ then the degree of $H_{|\lambda|,n}^{(\lambda)}$ is $n$ (see \cite{KM}).

Dur\'an \cite{D14} has found determinantal formulas for exceptional Hermite polynomials. Our new results for these polynomials will include an integral formula for $H_{2,n}^{(\{1,1\})}$ comparable to \eqref{lag3llms} (see Theorem~\ref{hermint}) and a determinantal formula for $H_{2,n}^{(\{1,1\})}$ that involves the zeros of the classical Hermite polynomials (see Theorem~\ref{hermdetform}). In this sense, our results are the natural extension of \cite{KLOa,KLO,LO} to the exceptional Hermite polynomials because we only consider the case when a small number of degrees are missing from the sequence.

\section{New formulas}\label{newform}

In this section we will present both integral and determinantal formulas for various families of exceptional orthogonal polynomials. All of the formulas and information about zeros that we need to complete our calculations was provided in the previous section.

\subsection[X\_m Laguerre polynomials]{$\boldsymbol{X_m}$ Laguerre polynomials}\label{xmlag}

We begin our analysis by stating the analog of \eqref{lag3llms} for the type-I and type-II exceptional $X_m$-Laguerre polynomials.

\begin{Theorem}\label{lagints}
The following formulas are valid:
\begin{itemize}\itemsep=0pt
\item[$(i)$] If $m\in\mathbb{N}$, $\alpha>0$, and $n\geq m$, then
\begin{gather*}
L_{m,n}^{I,(\alpha)}(x)=\frac{\alpha+n}{x^{\alpha}}\int_0^xt^{\alpha-1} L_m^{(\alpha-1)}(-t)L_{n-m}^{(\alpha-1)}(t)\,{\rm d}t.
\end{gather*}
\item[$(ii)$] If $m\in\mathbb{N}$, $\alpha>m-1$, and $n\geq m$, then
\begin{align*}
L_{m,n}^{II,(\alpha)}(x)=-(\alpha+n+1-2m){\rm e}^x\int_{x}^{\infty}{\rm e}^{-t}L_m^{(-\alpha-1)}(t)L_{n-m}^{(\alpha+1)}(t)\,{\rm d}t.
\end{align*}
\end{itemize}
\end{Theorem}

\begin{proof}
$(i)$ It follows from \cite[equation~(17)]{GMM} that if $y_n(x)=L_{m,n}^{I,(\alpha)}(x)$, then there exists a~polynomial $Q^I_{n-m}(x)$ so that
\begin{equation}\label{type1ode}
xy_n'(x)+\alpha y_n(x)=Q^I_{n-m}(x)L_{m}^{(\alpha-1)}(-x).
\end{equation}
Our first task is to show that $Q^I_{n-m}(x)=(\alpha+n)L_{n-m}^{(\alpha-1)}(x)$.

Starting from \eqref{xlag1form} and twice using \eqref{lagform1} we calculate
\begin{align*}
y_n'(x)={}&-L_m^{(\alpha)}(-x)L_{n-m-1}^{(\alpha)}(x)+L_{m-1}^{(\alpha+1)}(-x) L_{n-m}^{(\alpha-1)}(x)-L_{m}^{(\alpha-1)}(-x)L_{n-m-2}^{(\alpha+1)}(x)
\\
&+L_{m-1}^{(\alpha)}(-x) L_{n-m-1}^{(\alpha)}(x)\\
={}&L_{n-m-1}^{(\alpha)}(x)\big(L_{m-1}^{(\alpha)}(-x)-L_m^{(\alpha)}(-x)\big)
\\
&+L_{m-1}^{(\alpha+1)}(-x)L_{n-m}^{(\alpha-1)}(x)-L_{m}^{(\alpha-1)}(-x)L_{n-m-2}^{(\alpha+1)}(x)
\\
={}&-L_{n-m-1}^{(\alpha)}(x)L_m^{(\alpha-1)}(-x)+L_{m-1}^{(\alpha+1)}(-x) L_{n-m}^{(\alpha-1)}(x)-L_{m}^{(\alpha-1)}(-x)L_{n-m-2}^{(\alpha+1)}(x)
\\
={}&-L_m^{(\alpha-1)}(-x)\big(L_{n-m-1}^{(\alpha)}(x)+L_{n-m-2}^{(\alpha+1)}(x)\big) +L_{m-1}^{(\alpha+1)}(-x)L_{n-m}^{(\alpha-1)}(x)
\\
={}&-L_m^{(\alpha-1)}(-x)L_{n-m-1}^{(\alpha+1)}(x)+L_{m-1}^{(\alpha+1)}(-x)L_{n-m}^{(\alpha-1)}(x).
\end{align*}
Thus
\begin{align*}
xy_n'(x)+\alpha y_n(x)={}&L_{n-m}^{(\alpha-1)}(x)\big(xL_{m-1}^{(\alpha+1)}(-x)+\alpha L_m^{(\alpha)}(-x)\big)
\\
&+L_m^{(\alpha-1)}(-x)\big(\alpha L_{n-m-1}^{(\alpha)}(x)-xL_{n-m-1}^{(\alpha+1)}(x)\big).
\end{align*}
We use \eqref{lagform2} to rewrite the right-hand side of this equation as
\begin{gather*}
L_{n-m}^{(\alpha-1)}(x)\big(xL_{m-1}^{(\alpha+1)}(-x)+\alpha L_m^{(\alpha)}(-x)\big)+(n-m)L_m^{(\alpha-1)}(-x)\big(L_{n-m}^{(\alpha)}(x)-L_{n-m-1}^{(\alpha)}(x)\big).
\end{gather*}
Now we again employ \eqref{lagform1} to rewrite this as
\begin{gather*}
L_{n-m}^{(\alpha-1)}(x)\big(xL_{m-1}^{(\alpha+1)}(-x)+\alpha L_m^{(\alpha)}(-x)\big)+(n-m)L_m^{(\alpha-1)}(-x)L_{n-m}^{(\alpha-1)}(x)
\\ \qquad
{}=L_{n-m}^{(\alpha-1)}(x)\big(xL_{m-1}^{(\alpha+1)}(-x)+\alpha L_m^{(\alpha)}(-x)+(n-m)L_m^{(\alpha-1)}(-x)\big).
\end{gather*}
The proof of the desired formula for $Q_{n-m}^I$ will be complete if we can show that
\begin{gather*}
\alpha L_m^{(\alpha)}(t)-tL_{m-1}^{(\alpha+1)}(t)=(m+\alpha)L_{m}^{(\alpha-1)}(t),
\end{gather*}
which follows easily from \eqref{lagform1} and \eqref{lagform2}.

Consider now the first order linear ODE \eqref{type1ode} with $Q^I_{n-m}(x)=(\alpha+n)L_{n-m}^{(\alpha-1)}(x)$. Multiply both sides by $x^{\alpha-1}$ and integrate with respect to $x$ to obtain
\begin{gather*}
x^{\alpha}y_n(x)=(\alpha+n)\int_c^xt^{\alpha-1}L_m^{(\alpha-1)}(-t)L_{n-m}^{(\alpha-1)}(t)\,{\rm d}t
\end{gather*}
for some constant $c$ so that the left-hand side of this expression is equal to $0$ when evaluated at~$c$. Clearly, $c=0$ is a valid choice and the desired formula follows.

$(ii)$ Recall \cite[equation~(41)]{GMM}, which states
\begin{equation}\label{type2ode}
\frac{\rm d}{{\rm d}x}L_{m,n}^{II,(\alpha)}(x)-L_{m,n}^{II,(\alpha)}(x)=(\alpha+n+1-2m)L_{m}^{(-\alpha-1)}(x)L_{n-m}^{(\alpha+1)}(x)
\end{equation}
(note that we correct a typo in the constant from \cite[equation~(41)]{GMM}). Multiply both sides of this equation by ${\rm e}^{-x}$ and integrate to obtain
\begin{equation}%\label{lag2int}
{\rm e}^{-x}L_{m,n}^{II,(\alpha)}(x)=(\alpha+n+1-2m)\int_{c}^{x}{\rm e}^{-t}L_m^{(-\alpha-1)}(t)L_{n-m}^{(\alpha+1)}(t)\,{\rm d}t
\end{equation}
for a constant $c$ so that the left-hand side of this expression is equal to $0$ when evaluated at $c$. Clearly, taking the limit as $c\rightarrow\infty$ leads to a valid choice and the desired formula follows.
\end{proof}

We can use Theorem \ref{lagints} to prove the following determinantal formulas for the exceptional Laguerre polynomials. Different determinantal formulas for $L_{m,n}^{I,(\alpha)}$ ($m=1,2$) involving moments of an appropriate measure are presented in \cite{KLOa,KLO,LO}.

\begin{Theorem}\label{lag1detform}
For any $\tau\in\mathbb{C}$, let $\big\{x_{j,N}^{(\tau)}\big\}_{j=1}^N$ denote the zeros of $L_N^{(\tau)}$. Fix $m\in\mathbb{N}$ and $\alpha>0$ and let
\begin{gather*}
\big\{X^{(\alpha)}_{j,m,n}\big\}_{j=1}^n=\big\{{-}x_{j,m}^{(\alpha-1)}\big\}_{j=1}^m \cup\big\{x_{j-m,n-m}^{(\alpha-1)}\big\}_{j=m+1}^{n}.
\end{gather*}
%listed in increasing order.
Let $M_n^I$ be the $(n+1)\times(n+1)$ matrix whose $j^{th}$ row is
\begin{gather*}
\big[1,\big(X_{j,m,n}^{(\alpha)}\big),\big(X_{j,m,n}^{(\alpha)}\big)^2,\dots, \big(X_{j,m,n}^{(\alpha)}\big)^n\big]
\end{gather*}
for $j=1,\dots,n$ and let the last row of $M_n^I$ be
\begin{gather*}
\bigg[\frac{1}{\alpha},\frac{x}{1+\alpha},\frac{x^2}{2+\alpha},\dots,\frac{x^n}{n+\alpha}\bigg].
\end{gather*}
Then for $n\geq m$,
\begin{gather*}
L_{m,n}^{I,(\alpha)}(x)=\frac{(-1)^{n-m}(\alpha+n)}{m!(n-m)!\prod_{1\leq i<j\leq n}\big(X^{(\alpha)}_{j,m,n}-X^{(\alpha)}_{i,m,n}\big)}\det\big(M_n^I\big).
\end{gather*}
\end{Theorem}

\begin{Remark} Notice that the factor $\prod_{1\leq i<j\leq n}\big(X^{(\alpha)}_{j,m,n}-X^{(\alpha)}_{i,m,n}\big)$ is the Vandermonde determinant corresponding to the points $\big\{X^{(\alpha)}_{j,m,n}\big\}_{j=1}^n$. We will see similar formulas in our later results.
\end{Remark}

\begin{proof}Write
\begin{gather*}
L_{m,n}^{I,(\alpha)}(x)=\sum_{j=0}^na_{j,n}x^j.
\end{gather*}
Let $\tilde{M}_n^I$ be the $(n+1)\times(n+1)$ matrix whose $j^{th}$ row is
\begin{gather*}
\big[\alpha,\big(X_{j,m,n}^{(\alpha)}\big)(1+\alpha),\big(X_{j,m,n}^{(\alpha)}\big)^2(2+\alpha), \dots,\big(X_{j,m,n}^{(\alpha)}\big)^n(n+\alpha)\big]
\end{gather*}
for $j=1,\dots,n$ and let the last row of $\tilde{M}_n^I$ be $\big[1,x,x^2,\dots,x^n\big]$.
Equation \eqref{type1ode} shows that $\big\{X^{(\alpha)}_{j,m,n}\big\}_{j=1}^n$ is the zero set of $xL_{m,n}^{I,(\alpha)}(x)'+\alpha L_{m,n}^{I,(\alpha)}(x)$. Therefore, if
\begin{gather*}
\vec{a}=\begin{pmatrix}a_{0,n}\\ a_{1,n}\\ \vdots\\ a_{n,n}\end{pmatrix}\!,\qquad
\vec{b}_I=\begin{pmatrix}0\\ \vdots\\ 0\\ L_{m,n}^{I,(\alpha)}(x)\end{pmatrix}\!,
\end{gather*}
then $\tilde{M}_n^I\vec{a}=\vec{b}_I$. It is clear that the determinant of $\tilde{M}_n^I$ is a polynomial of degree (at most)~$n$ and the relation $\tilde{M}_n^I\vec{a}=\vec{b}_I$ tells us that if we plug in any zero of $L_{m,n}^{I,(\alpha)}(x)$ for $x$, then $\vec{a}\in\operatorname{Ker}\big(\tilde{M}_n^I\big)$ and so $\det\big(\tilde{M}_n^I\big)=0$ at that point. We deduce that $\det\big(\tilde{M}_n^I\big)=C'L_{m,n}^{I,(\alpha)}(x)$ for some $C'\in\mathbb{C}$. It~is clear that $M_n^I$ is obtained from $\tilde{M}_n^I$ by elementary column operations. Therefore, $\det\big(M_n^I\big)=CL_{m,n}^{I,(\alpha)}(x)$ for some $C\in\mathbb{C}$.

To determine this constant $C$, notice that the coefficient of $x^n$ in $\det(M_n^I)$ is
\begin{gather*}
\det\begin{pmatrix}
1 & X^{(\alpha)}_{1,m,n} & \big(X^{(\alpha)}_{1,m,n}\big)^2 & \cdots & \big(X^{(\alpha)}_{1,m,n}\big)^{n-1}\\
1 & X^{(\alpha)}_{2,m,n} & \big(X^{(\alpha)}_{2,m,n}\big)^2 & \cdots & \big(X^{(\alpha)}_{2,m,n}\big)^{n-1}\\
\vdots & \vdots &\vdots & \ddots & \vdots\\
1 & X^{(\alpha)}_{n,m,n} & \big(X^{(\alpha)}_{n,m,n}\big)^2 & \cdots & \big(X^{(\alpha)}_{n,m,n}\big)^{n-1}
\end{pmatrix}\!\cdot\frac{1}{\alpha\!+\!n}
=\frac{1}{\alpha\!+\!n}\!\prod_{1\leq i<j\leq n}\!\!\big(X^{(\alpha)}_{j,m,n}-X^{(\alpha)}_{i,m,n}\big).
\end{gather*}
Observe that the points $X_{j,m,n}^{(\alpha)}$ are distinct by part $(i)$ of Theorem~\ref{clagzero}. The desired formula now follows from \eqref{lagform} and \eqref{xlag1form}.
\end{proof}

For our next result, define the polynomials $\{E_n(x)\}_{n=0}^{\infty}$ by
\begin{gather*}
E_n(x)=\sum_{j=0}^n\frac{x^j}{j!}.
\end{gather*}
These are the partial sums of the Taylor series for ${\rm e}^x$ around zero and it is also true that $E_n(x)=(-1)^nL_n^{(-n-1)}(x)$ (see \cite[Section 1.2]{KM01}).

\begin{Theorem}\label{lag2detform}
For any $\tau\in\mathbb{C}$, let $\big\{x_{j,N}^{(\tau)}\big\}_{j=1}^N$ denote the zeros of $L_N^{(\tau)}$. Fix $m\in\mathbb{N}$ and $\alpha\in(m-1,\infty)$ and let
\begin{gather*}
\big\{Y^{(\alpha)}_{j,m,n}\big\}_{j=1}^n=\big\{x_{j,m}^{(-\alpha-1)}\big\}_{j=1}^m \cup\big\{x_{j-m,n-m}^{(\alpha+1)}\big\}_{j=m+1}^{n}.
\end{gather*}
Let $M_n^{II}$ be the $(n+1)\times(n+1)$ matrix whose $j^{th}$ row is
\begin{gather*}
\big[1,Y_{j,m,n}^{(\alpha)},\big(Y_{j,m,n}^{(\alpha)}\big)^2,\dots,\big(Y_{j,m,n}^{(\alpha)}\big)^{n}\big]
\end{gather*}
for $j=1,\dots,n$ and let the last row of $M_n^{II}$ be
\begin{gather*}
[E_0(x),E_1(x),2!E_2(x),3!E_3(x),\dots,n!E_n(x)].
\end{gather*}
Then for $n\geq m$,
\begin{gather*}
L_{m,n}^{II,(\alpha)}(x)=\frac{(-1)^{n+1}-(m-\alpha-1)/(n-m)}{m!(n-m-1)!\prod_{1\leq i<j\leq n}\big(Y^{(\alpha)}_{j,m,n}-Y^{(\alpha)}_{i,m,n}\big)}\det\big(M_n^{II}\big).
\end{gather*}
\end{Theorem}

\begin{proof}Write
\begin{gather*}
L_{m,n}^{II,(\alpha)}(x)=\sum_{j=0}^nc_{j,n}x^j.
\end{gather*}
Let $\tilde{M}_n^{II}$ be the $(n+1)\times(n+1)$ matrix whose $j^{th}$ row is
\begin{gather*}
\big[1,(Y_{j,m,n}^{(\alpha)}-1),\big(Y_{j,m,n}^{(\alpha)}\big) (Y_{j,m,n}^{(\alpha)}-2),\dots,\big(Y_{j,m,n}^{(\alpha)}\big)^{n-1}\big(Y_{j,m,n}^{(\alpha)}-n\big)\big]
\end{gather*}
for $j=1,\dots,n$ and let the last row of $\tilde{M}_n^{II}$ be $\big[1,x,x^2,\dots,x^n\big]$. Equation \eqref{type2ode} implies $\big\{Y^{(\alpha)}_{j,m,n}\big\}_{j=1}^n$ is the zero set of $L_{m,n}^{II,(\alpha)}(x)'- L_{m,n}^{II,(\alpha)}(x)$.
Therefore, if
\begin{gather*}
\vec{c}=\begin{pmatrix}c_{0,n}\\ c_{1,n}\\ \vdots\\ c_{n,n}\end{pmatrix}\!,\qquad
\vec{b}_{II}=\begin{pmatrix}0\\ \vdots\\ 0\\ L_{m,n}^{II,(\alpha)}(x)\end{pmatrix}\!,
\end{gather*}
then $\tilde{M}_n^{II}\vec{c}=\vec{b}_{II}$. Reasoning as in the proof of Theorem \ref{lag1detform}, we deduce that $\det\big(M_n^{II}\big)=CL_{m,n}^{II,(\alpha)}(x)$ for some $C\in\mathbb{C}$.

To determine this constant $C$, notice that the coefficient of $x^n$ in $\det\big(M_n^{II}\big)$ is
\begin{align*}
&\det\begin{pmatrix}
1 & Y^{(\alpha)}_{1,m,n} & \big(Y^{(\alpha)}_{1,m,n}\big)^2 & \cdots & \big(Y^{(\alpha)}_{1,m,n}\big)^{n-1}\\[1ex]
1 & Y^{(\alpha)}_{2,m,n} & \big(Y^{(\alpha)}_{2,m,n}\big)^2 & \cdots & \big(Y^{(\alpha)}_{2,m,n}\big)^{n-1}\\
\vdots & \vdots &\vdots & \ddots & \vdots\\
1 & Y^{(\alpha)}_{n,m,n} & \big(Y^{(\alpha)}_{n,m,n}\big)^2 & \cdots & \big(Y^{(\alpha)}_{n,m,n}\big)^{n-1}
\end{pmatrix}=\prod_{1\leq i<j\leq n}\big(Y^{(\alpha)}_{j,m,n}-Y^{(\alpha)}_{i,m,n}\big).
\end{align*}
The desired formula now follows from \eqref{lagform}, from \eqref{xlag2form}, and from Theorem \ref{clagzero}, which guarantees that the set $\big\{Y_{j,m,n}^{(\alpha)}\big\}_{j=1}^n$ consists of $n$ distinct points.
\end{proof}

Our next result is a comparable determinantal formula for the type-III exceptional Laguerre polynomials. The formula that we will derive is slightly less elegant than those derived in Theorems \ref{lag1detform} and \ref{lag2detform} because we will only have information about zeros of the derivative of~$L_{m,n}^{III,(\alpha)}$ and hence will have to account for $L_{m,n}^{III,(\alpha)}(0)$ separately.

\begin{Theorem}\label{lag3detform}
For any $\tau\in\mathbb{C}$, let $\big\{x_{j,N}^{(\tau)}\big\}_{j=1}^N$ denote the zeros of $L_N^{(\tau)}$. Fix $m\in\mathbb{N}$ and $\alpha\in(-1,0)$ and let
\begin{gather*}
\big\{Z^{(\alpha)}_{j,m,n}\big\}_{j=1}^{n-1}=\big\{{-}x_{j,m}^{(-\alpha-1)}\big\}_{j=1}^m \cup\big\{x_{j-m,n-m-1}^{(\alpha+1)}\big\}_{j=m+1}^{n-1}.
\end{gather*}
Let $M_n^{III}$ be the $n\times n$ matrix whose $j^{th}$ row is
\begin{gather*}
\big[1,\big(Z_{j,m,n}^{(\alpha)}\big),\big(Z_{j,m,n}^{(\alpha)}\big)^2, \dots,\big(Z_{j,m,n}^{(\alpha)}\big)^{n-1}\big]
\end{gather*}
for $j=1,\dots,n-1$ and let the last row of $M_n^{III}$ be
\begin{gather*}
\bigg[1,\frac{x}{2},\dots,\frac{x^{n-1}}{n}\bigg].
\end{gather*}
Then for $n>m$,
\begin{align*}
L_{m,n}^{III,(\alpha)}(x)={}&\frac{x(-1)^{n-m-1}n\det\big(M_n^{III}\big)}{m!(n-m-1)!\prod_{1\leq i<j\leq n-1}\big(Z^{(\alpha)}_{j,m,n}-Z^{(\alpha)}_{i,m,n}\big)}
\\
&+(m+1)\binom{n-m+\alpha}{n-m-1}\binom{m-\alpha-1}{m+1}.
\end{align*}
\end{Theorem}

\begin{proof}Write
\begin{gather*}
L_{m,n}^{III,(\alpha)}(x)=\sum_{j=0}^nh_{j,n}x^j.
\end{gather*}
Let $\tilde{M}_n^{III}$ be the $n\times n$ matrix whose $j^{th}$ row is
\begin{gather*}
\big[1,2Z_{j,m,n}^{(\alpha)},3\big(Z_{j,m,n}^{(\alpha)}\big)^2,\dots,n
\big(Z_{j,m,n}^{(\alpha)}\big)^{n-1}\big]
\end{gather*}
for $j=1,\dots,n-1$ and let the last row of $\tilde{M}_n^{III}$ be $\big[1,x,x^2,\dots,x^{n-1}\big]$. Equation~\eqref{lag3llms} implies $\big\{Z^{(\alpha)}_{j,m,n}\big\}_{j=1}^{n-1}$ is the zero set of $L_{m,n}^{III,(\alpha)}(x)'$.
Therefore, if $P_n(x)=\big(L_{m,n}^{III,(\alpha)}(x)-h_{0,n}\big)/x$ and
\begin{gather*}
\vec{h}=\begin{pmatrix}h_{1,n}\\ h_{2,n}\\ \vdots\\ h_{n,n}\end{pmatrix}\!,\qquad
\vec{b}_{III}=\begin{pmatrix}0\\ \vdots\\ 0\\ P_n(x)\end{pmatrix}\!,
\end{gather*}
then $\tilde{M}_n^{III}\vec{h}=\vec{b}_{III}$. Reasoning as in the proof of Theorem \ref{lag1detform}, we deduce that $\det\big(M_n^{III}\big)=CP_n(x)$ for some $C\in\mathbb{C}$.

To determine this constant $C$, notice that the coefficient of $x^{n-1}$ in $\det\big(M_n^{III}\big)$ is
\begin{gather*}
\frac{1}{n} \det\begin{pmatrix}
1 & Z^{(\alpha)}_{1,m,n} & \big(Z^{(\alpha)}_{1,m,n}\big)^2 & \cdots & \big(Z^{(\alpha)}_{1,m,n}\big)^{n-2}\\
1 & Z^{(\alpha)}_{2,m,n} & \big(Z^{(\alpha)}_{2,m,n}\big)^2 & \cdots & \big(Z^{(\alpha)}_{2,m,n}\big)^{n-2}\\
\vdots & \vdots &\vdots & \ddots & \vdots\\
1 & Z^{(\alpha)}_{n-1,m,n} & \big(Z^{(\alpha)}_{n-1,m,n}\big)^2 & \cdots & \big(Z^{(\alpha)}_{n-1,m,n}\big)^{n-2}
\end{pmatrix}\!
 =\frac{1}{n}\!\prod_{1\leq i<j\leq n-1}\!\!\!\big(Z^{(\alpha)}_{j,m,n}\!-Z^{(\alpha)}_{i,m,n}\big).
\end{gather*}
The desired formula now follows from \eqref{lagform}, from \eqref{xlag3form}, from Theorem \ref{clagzero} (which guarantees that the set $\big\{Z_{j,m,n}^{(\alpha)}\big\}_{j=1}^{n-1}$ consists of ${n-1}$ distinct points), and the formula for $L_{m,n}^{III,(\alpha)}(0)$ from~\eqref{lag3llms}.
\end{proof}

\subsection[X\_m Jacobi polynomials]{$\boldsymbol{X_m}$ Jacobi polynomials}%\label{xmjac}

In this section, we will present results for $X_m$ Jacobi polynomials that are analogous to those in Section \ref{xmlag}.

\begin{Theorem}\label{jacint}
Fix $m\in\mathbb{N}$ and let $P_{m,n}^{(\alpha,\beta)}(x)$ denote the degree $n$ $X_m$ Jacobi polynomial with $\beta>0$, $\alpha>m-1$, and $\alpha+1-m-\beta\not\in\{0,1,\dots,m-1\}$. Then
\begin{gather*}
P_{m,n}^{(\alpha,\beta)}(x)=(-1)^{m}\frac{(\beta+n)(\alpha+n-2m+1)}{(\alpha+n-m+1)(x+1)^{\beta}} \int_{-1}^x(t+1)^{\beta-1}P_m^{(-\alpha-1,\beta-1)}(t)P_{n-m}^{(\alpha+1,\beta-1)}(t)\,{\rm d}t
\end{gather*}
for all $n\geq m$.
\end{Theorem}

\begin{proof}
Recall \cite[equation~(70)]{GMM} states that
\begin{gather}
(1+x)P_{m,n}^{(\alpha,\beta)}(x)'+\beta P_{m,n}^{(\alpha,\beta)}(x)\nonumber
\\ \qquad
{}=(-1)^{m}\frac{(\beta+n)(\alpha+n-2m+1)}{(\alpha+n-m+1)}P_{n-m}^{(\alpha+1,\beta-1)}(x)P^{(-\alpha-1,\beta-1)}_m(x).
\label{xjaccond}
\end{gather}
Multiply both sides of \eqref{xjaccond} by $(x+1)^{\beta-1}$ and integrate to obtain
\begin{gather*}
(x+1)^{\beta}P_{m,n}^{(\alpha,\beta)}(x)
\\ \qquad
{}=(-1)^{m}\frac{(\beta+n)(\alpha+n-2m+1)}{(\alpha+n-m+1)}\int_{c}^x(t+1)^{\beta-1}P_m^{(-\alpha-1,\beta-1)}(t)P_{n-m}^{(\alpha+1,\beta-1)}(t)\,{\rm d}t
\end{gather*}
for a constant $c$ so that the left-hand side of this expression is equal to $0$ when evaluated at $c$. Clearly, $c=-1$ is a valid choice (because $\beta>0$) and the desired formula follows.
\end{proof}

For our next result, define the polynomials $\{R_n(x)\}_{n=0}^{\infty}$ by
\begin{gather*}
R_n(x;\beta)=\sum_{j=0}^n\frac{\beta^{(j)}(-x)^{j}}{j!}.
\end{gather*}
These polynomials are the partial sums of the Taylor series for $(1+x)^{-\beta}$ around zero and it is also true that $R_n(x,\beta)=(-1)^nP_{n}^{(-n-1,\beta)}(1+2x)$. In our notation, we used the Pochhammer symbol
\begin{gather*}
a^{(n)}=\begin{cases}
a(a+1)(a+2)\cdots(a+n-1)&\text{if} \ n>0,\\
1&\text{if} \ n=0.
\end{cases}
\end{gather*}

\begin{Theorem}\label{jacdetform}
Let $\big\{z_{j,N}^{(\alpha,\beta)}\big\}_{j=1}^N$ denote the zeros of $P_N^{(\alpha,\beta)}$. Choose $\alpha$, $\beta$, $m$ so that the hypotheses of Theorem~$\ref{jacint}$ are satisfied. Let
\begin{gather*}
\big\{W^{(\alpha,\beta)}_{j,m,n}\big\}_{j=1}^n=\big\{z_{j,m}^{(-\alpha-1,\beta-1)}\big\}_{j=1}^m \cup\big\{z_{j-m,n-m}^{(\alpha+1,\beta-1)}\big\}_{j=m+1}^{n}.
\end{gather*}
Let $M_n$ be the $(n+1)\times(n+1)$ matrix whose $j^{th}$ row is
\begin{gather*}
\big[1,W_{j,m,n}^{(\alpha,\beta)},\big(W_{j,m,n}^{(\alpha,\beta)}\big)^2,\dots, \big(W_{j,m,n}^{(\alpha,\beta)}\big)^{n}\big]
\end{gather*}
for $j=1,\dots,n$ and let the last row of $M_n$ be
\begin{gather*}
\bigg[\frac{1}{\beta^{(1)}},\frac{-1}{\beta^{(2)}}R_1(x),\frac{2!}{\beta^{(3)}}R_2(x), \dots,\frac{(-1)^nn!}{\beta^{(n+1)}}R_n(x)\bigg].
\end{gather*}
Then for $n\geq m$ $($with $\delta=\alpha+1)$,
\begin{gather*}
P_{m,n}^{(\alpha,\beta)}(x)=\frac{(-1)^m2^{-n}\binom{n}{m}(n\!+\!\delta\!-\!2m)\Gamma(\beta\!-\!\delta\!+\!2m) \Gamma(\delta\!+\!\beta\!+\!2(n\!-\!m))(n\!+\!\beta)} {n!(\delta+n-m)\Gamma(\beta\!-\!\delta\!+\!m)\Gamma(\delta\!+\!\beta\!+\!n\!-\!m) \prod_{i<j}\big(W^{(\alpha,\beta)}_{j,m,n} -W^{(\alpha,\beta)}_{i,m,n}\big)}\det(M_n).
\end{gather*}
\end{Theorem}

\begin{proof}
Write
\begin{gather*}
P_{m,n}^{(\alpha,\beta)}(x)=\sum_{j=0}^nd_{j,n}x^j.
\end{gather*}
Let $\tilde{M}_n$ be the $(n+1)\times(n+1)$ matrix whose $j^{th}$ row is
\begin{gather*}
\big[\beta,\big(1\!+(1\!+\beta)W_{j,m,n}^{(\alpha,\beta)}\big),\big(W_{j,m,n}^{(\alpha,\beta)}\big)\big(2\!+(2\!+\beta) W_{j,m,n}^{(\alpha,\beta)}\big),\dots,\big(W_{j,m,n}^{(\alpha,\beta)}\big)^{n-1} \big(n\!+(n\!+\beta)W_{j,m,n}^{(\alpha,\beta)}\big)\big]
\end{gather*}
{\sloppy for $j=1,\dots,n$ and let the last row of $\tilde{M}_n$ be $\big[1,x,x^2,\dots,x^n\big]$.
Equation \eqref{xjaccond} implies $\big\{W^{(\alpha)}_{j,m,n}\big\}_{j=1}^n$ is the zero set of $(1+x)P_{m,n}^{(\alpha,\beta)}(x)'+\beta P_{m,n}^{(\alpha,\beta)}(x)$. Therefore, if
\begin{gather*}
\vec{d}=\begin{pmatrix}d_{0,n}\\ d_{1,n}\\ \vdots\\ d_{n,n}\end{pmatrix}\!,\qquad
\vec{b}=\begin{pmatrix}0\\ \vdots\\ 0\\ P_{m,n}^{(\alpha,\beta)}(x)\end{pmatrix}\!,
\end{gather*}}\noindent
then $\tilde{M}_n\vec{d}=\vec{b}$. Reasoning as in the proof of Theorem \ref{lag1detform}, we deduce that $\det(M_n)=CP_{m,n}^{(\alpha,\beta)}(x)$ for some $C\in\mathbb{C}$.

To determine this constant $C$, notice that the coefficient of $x^n$ in $\det(M_n)$ is
\begin{gather*}
\frac{1}{n+\beta}\det\begin{pmatrix}
1 & W^{(\alpha,\beta)}_{1,m,n} & \big(W^{(\alpha,\beta)}_{1,m,n}\big)^2 & \cdots & \big(W^{(\alpha,\beta)}_{1,m,n}\big)^{n-1}\\
1 & W^{(\alpha,\beta)}_{2,m,n} & \big(W^{(\alpha,\beta)}_{2,m,n}\big)^2 & \cdots & \big(W^{(\alpha,\beta)}_{2,m,n}\big)^{n-1}\\
\vdots & \vdots &\vdots & \ddots & \vdots\\
1 & W^{(\alpha,\beta)}_{n,m,n} & \big(W^{(\alpha,\beta)}_{n,m,n}\big)^2 & \cdots & \big(W^{(\alpha,\beta)}_{n,m,n}\big)^{n-1}
\end{pmatrix}\!=\!\frac{1}{n\!+\!\beta}\!\prod_{1\leq i<j\leq n}\!\!\!\!\big(W^{(\alpha,\beta)}_{j,m,n}-W^{(\alpha,\beta)}_{i,m,n}\big).
\end{gather*}
The desired formula now follows from \eqref{jacform}, from \eqref{xjacform}, and from Theorem \ref{cjaczero}, which guarantees us that the set $\big\{W_{j,m,n}^{(\alpha,\beta)}\big\}_{j=1}^n$ consists of $n$ distinct points.
\end{proof}

\subsection[X\_2 Hermite polynomials]{$\boldsymbol{X_2}$ Hermite polynomials}%\label{xmherm}

The starting point for our work with these polynomials comes from \cite[Proposition 5.4]{GGM14}, which states that if a polynomial $p$ is in the span of $\big\{H_{|\lambda|,n}^{(\lambda)}\big\}_{n\in\mathbb{N}_{\lambda}}$, then
\begin{gather*}
2H_{\lambda}'(xp-p')+H_{\lambda}''p
\end{gather*}
is divisible by $H_{\lambda}$ (note that the complete statement of \cite[Proposition 5.4]{GGM14} was corrected in \cite{D20,GGM14a}). Since each polynomial in the sequence $\big\{H_{|\lambda|,n}^{(\lambda)}\big\}_{n\in\mathbb{N}_{\lambda}}$ is obviously in the span of that sequence, it is interesting to find the polynomial $Q_{n,\lambda}$ such that
\begin{equation}\label{hermq}
2H_{\lambda}'(x)\big(xH_{|\lambda|,n}^{(\lambda)}(x)-H_{|\lambda|,n}^{(\lambda)}(x)'\big) +H_{\lambda}''(x)H_{|\lambda|,n}^{(\lambda)}(x)=Q_{n,\lambda}(x)H_{\lambda}(x).
\end{equation}
The search for such a polynomial leads us to the following result.

\begin{Theorem}\label{hermint}
Let $\lambda=\{1,1\}$ and let $H_{2,n}^{(\{1,1\})}$ denote the corresponding degree $n$ exceptional Hermite polynomial with $n\geq3$. Then
\begin{align*}
H_{2,n}^{(\{1,1\})}(x)&=8n(n-1)(n-2)\int_0^xH_{\{1,1\}}(t)H_{n-3}(t)\,{\rm d}t+16(n-1)(n-2)H_{n-2}(0)
\\
&=\int_0^xH_{\{1,1\}}(t)H'''_{n}(t)\,{\rm d}t+8(n-2)H_{n-1}'(0).
\end{align*}
\end{Theorem}

\begin{proof}
It suffices to show that
\begin{gather*}
\frac{\rm d}{{\rm d}x}H_{2,n}^{(\{1,1\})}(x)=8n(n-1)(n-2)H_{\{1,1\}}(x)H_{n-3}(x),
\end{gather*}
 and to calculate $H_{2,n}^{(\{1,1\})}(0)$. This follows easily from the formulas in \cite[Section 6.4]{SimExcep}, namely $H_{\{1,1\}}(x)=4\big(2x^2+1\big)$ and
\begin{gather*}
H_{2,n}^{(\{1,1\})}(x)=\det\begin{pmatrix}
2x & 4x^2-2 & H_n(x)\\
2 & 8x & H_n'(x)\\
0 & 8 & H_n''(x)
\end{pmatrix}
\end{gather*}
for $n\in\mathbb{N}_{\{1,1\}}$. As mentioned there, the formulas \eqref{hermder} and \eqref{hermrecur} imply
\begin{equation}\label{xhermform}
H_{2,n}^{(\{1,1\})}(x)=16(n-1)\big({-}2xH_{n-1}(x)+\big(n\big(2x^2+1\big)-2\big)H_{n-2}(x)\big)
\end{equation}
and from here, the first equality in the theorem is an elementary calculation using \eqref{hermder} and~\eqref{hermrecur}. The second formula follows from \eqref{hermder}.
\end{proof}

Numerical experiments have confirmed that for partitions other than $\{1,1\}$, it is not necessarily true that $H_{\lambda}$ divides $H_{|\lambda|,n}^{(\lambda)}(x)'$. One can use the result of \cite[Proposition 5.4]{GGM14} mentioned above to find a first order differential relation for $H_{|\lambda|,n}^{(\lambda)}(x)$, but the coefficients will be left in terms of $H_{\lambda}$, which is difficult to calculate explicitly and hence does not lead to a nice integral formula as in Theorem \ref{hermint}. Even still, it would be interesting to characterize the polynomials~$Q_{n,\lambda}$ from \eqref{hermq}.

Following the method of the previous sections, we can use Theorem \ref{hermint} to prove a determinantal formula for $H_{2,n}^{(\{1,1\})}$.

\begin{Theorem}\label{hermdetform}
Let $\{y_{j,N}\}_{j=1}^N$ denote the zeros of the classical Hermite polynomial $H_N$ and let
\begin{gather*}
\{U_{j,n}\}_{j=1}^{n-1}=\bigg\{\frac{\rm i}{\sqrt{2}},\frac{-{\rm i}}{\sqrt{2}}\bigg\} \cup\{y_{j-2,n-3}\}_{j=3}^{n-1}.
\end{gather*}
Let $M_n^{\{1,1\}}$ be the $n\times n$ matrix whose $j^{th}$ row is
\begin{gather*}
\big[1,(U_{j,n}),(U_{j,n})^2,\dots,(U_{j,n})^{n-1}\big]
\end{gather*}
for $j=1,\dots,n-1$ and let the last row of $M_n^{\{1,1\}}$ be
\begin{gather*}
\bigg[1,\frac{x}{2},\dots,\frac{x^{n-1}}{n}\bigg].
\end{gather*}
Then for $n\geq3$
\begin{gather*}
H_{2,n}^{(\{1,1\})}(x)=x\,\frac{2^{n+3}n(n-1)(n-2)}{\prod_{1\leq i<j\leq n-1}(U_{j,n}-U_{i,n})}\det\big(M_n^{\{1,1\}}\big)+16(n-1)(n-2)H_{n-2}(0).
\end{gather*}
\end{Theorem}

\begin{Remark} Notice the similarity to Theorem \ref{lag3detform} in that the constant term of the polynomial needs to be added in separately because the determinant alone only gives a polynomial of degree~$n-1$.
\end{Remark}

\begin{proof}
Write
\begin{gather*}
H_{2,n}^{(\{1,1\})}(x)=\sum_{j=0}^nt_{j,n}x^j.
\end{gather*}
Let $\tilde{M}_n^{\{1,1\}}$ be the $n\times n$ matrix whose $j^{th}$ row is
\begin{gather*}
\big[1,2U_{j,n},3(U_{j,n})^2,\dots,n(U_{j,n})^{n-1}\big]
\end{gather*}
for $j=1,\dots,n-1$ and let the last row of $\tilde{M}_n^{\{1,1\}}$ be $\big[1,x,x^2,\dots,x^{n-1}\big]$. Theorem~\ref{hermint} implies $\{U_{j,n}\}_{j=1}^{n-1}$ is the zero set of $H_{2,n}^{(\{1,1\})}(x)'$.
Therefore, if $Q_n(x)=\big(H_{2,n}^{(\{1,1\})}(x)-t_{0,n}\big)/x$ and
\begin{gather*}
\vec{t}=\begin{pmatrix}t_{1,n}\\ t_{2,n}\\ \vdots\\ t_{n,n}\end{pmatrix}\!,\qquad
\vec{b}_{\{1,1\}}=\begin{pmatrix}0\\ \vdots\\ 0\\ Q_n(x)\end{pmatrix}\!,
\end{gather*}
then $\tilde{M}_n^{\{1,1\}}\vec{t}=\vec{b}_{\{1,1\}}$. Reasoning as in the proof of Theorem \ref{lag1detform}, we deduce that $\det\big(M_n^{\{1,1\}}\big)=CQ_n(x)$ for some $C\in\mathbb{C}$.

To determine this constant $C$, notice that the coefficient of $x^{n-1}$ in $\det\big(M_n^{\{1,1\}}\big)$ is
\begin{gather*}
\frac{1}{n}\,\det\begin{pmatrix}
1 & U_{1,n} & (U_{1,n})^2 & \cdots & (U_{1,n})^{n-2}\\
1 & U_{2,n} & (U_{2,n})^2 & \cdots & (U_{2,n})^{n-2}\\
\vdots & \vdots &\vdots & \ddots & \vdots\\
1 & U_{n-1,n} & (U_{n-1,n})^2 & \cdots & (U_{n-1,n})^{n-2}
\end{pmatrix}=\frac{1}{n}\prod_{1\leq i<j\leq n-1}(U_{j,n}-U_{i,n}).
\end{gather*}
The desired formula now follows from \eqref{hermlag} and~\eqref{xhermform}.
\end{proof}

\subsection*{Acknowledgements}

It is a pleasure to thank Rob Milson for helpful discussion about the content of this work. The author gratefully acknowledges support from the Simons Foundation through collaboration grant 707882. The author also thanks the anonymous referees for their useful feedback.

\pdfbookmark[1]{References}{ref}
\LastPageEnding

\end{document}